\documentclass[11pt,reqno]{amsart}
\usepackage{geometry}                
\usepackage[svgnames]{xcolor}
\usepackage{tikz}
\usetikzlibrary{decorations.markings}
\usetikzlibrary{shapes.geometric}
\usepackage{mathrsfs}
\usepackage{amssymb}
\usepackage{amsfonts}
\usepackage{algorithm}
\usepackage{amsthm}
\usepackage{amsmath, multicol}
\pgfdeclarelayer{edgelayer} \pgfdeclarelayer{nodelayer}
\pgfsetlayers{edgelayer,nodelayer,main}
\usepackage{enumitem}
\tikzstyle{none}=[inner sep=0pt]
\definecolor{hexcolor0xffffff}{rgb}{1.000,1.000,1.000}
\definecolor{hexcolor0x000000}{rgb}{0.000,0.000,0.000}
\definecolor{hexcolor0x00ff00}{rgb}{0.000,1.000,0.000}
\definecolor{hexcolor0xffff00}{rgb}{1.000,1.000,0.000}
\definecolor{hexcolor0x000000}{rgb}{0.000,0.000,0.000}
\definecolor{hexcolor0x000000}{rgb}{0.000,0.000,0.000}
\definecolor{hexcolor0x0000ff}{rgb}{0.000,0.000,1.000}
\definecolor{hexcolor0x000000}{rgb}{0.000,0.000,0.000}
\definecolor{hexcolor0xdc143c}{rgb}{0.863,0.078,0.235}
\definecolor{hexcolor0x32cd32}{rgb}{0.196,0.804,0.196}

\tikzstyle{rn}=[circle,fill=hexcolor0xffffff,draw=hexcolor0x000000,line width=0.8 pt]
\tikzstyle{gn}=[circle,fill=hexcolor0x00ff00,draw=hexcolor0x000000,line width=0.8 pt]
\tikzstyle{yn}=[circle,fill=hexcolor0xffff00,draw=hexcolor0x000000,line width=0.8 pt]

\tikzstyle{simple}=[-,draw=hexcolor0x000000,line width=2.000]
\tikzstyle{blue}=[-,draw=hexcolor0x0000ff,postaction={decorate},decoration={markings,mark=at
position .5 with {\arrow{>}}},line width=2.000]
\tikzstyle{tick}=[-,draw=hexcolor0x000000,postaction={decorate},decoration={markings,mark=at
position .5 with {\draw (0,-0.1) -- (0,0.1);}},line width=2.000]
\tikzstyle{black}=[->,draw=hexcolor0x000000,line width=2.000]\tikzstyle{red}=[-,draw=hexcolor0xdc143c,postaction={decorate},decoration={markings,mark=at
position .5 with {\arrow{>}}},line width=2.000]
\tikzstyle{green}=[-,draw=hexcolor0x32cd32,postaction={decorate},decoration={markings,mark=at
position .5 with {\arrow{>}}},line width=2.000]
\tikzstyle{match}=[-,draw=hexcolor0xdc143c]

\newtheorem{Theorem}{Theorem}[section]
\newtheorem{Lemma}[Theorem]{Lemma}

\newtheorem{Prop}[Theorem]{Proposition}

\newcommand{\QED}{\ \hfill \rule{0.5em}{0.5em} }
\newcommand{\omcos}[1]{$(H#1)^\omega$}

\begin{document}

\title{Graph immersions, inverse monoids and deck transformations}

\author{Corbin Groothuis and John Meakin}

\maketitle

\begin{abstract}


If $f : \tilde{\Gamma} \rightarrow \Gamma$ is a  covering  map
between connected graphs, and $H$ is the subgroup of
$\pi_1(\Gamma,v)$ used to construct the cover, then it is well known
that the group  of deck transformations of the cover is isomorphic
to $ N(H)/H$, where $N(H)$ is the normalizer of $H$ in
$\pi_1(\Gamma,v)$. We show that an entirely analogous result holds
for  immersions between connected graphs, where the subgroup $H$ is
replaced by the closed inverse submonoid of the inverse monoid
$L(\Gamma,v)$ used to construct the immersion. We observe a
relationship between group actions on graphs and deck
transformations of graph immersions. We also show that a graph
immersion $f : \tilde{\Gamma} \rightarrow \Gamma$ may be extended to
a cover $g : \tilde{\Delta} \rightarrow \Gamma$ in such a way that
all deck transformations of $f$ are restrictions of deck
transformations of $g$.

\medskip

\noindent AMS subject numbers: 20M18, 57M10


\end{abstract}

\section
{Introduction}

It is well known that group theory provides a powerful algebraic
tool for studying covering spaces of topological spaces. For
example, under mild conditions on a connected topological space
$\mathcal X$, the connected covers of $\mathcal X$ may be classified
via subgroups of the fundamental group of $\mathcal X$. This may be
used to study deck transformations of covering spaces and actions of
groups on topological spaces. However, the study of {\em immersions}
between connected topological spaces seems to require somewhat
different algebraic tools, even for graphs (1-dimensional
$CW$-complexes).

In his paper \cite{Stall}, Stallings made use of immersions between
finite graphs to study finitely generated subgroups of free groups.
Here by an immersion between graphs we mean a locally injective
graph morphism, that is, a graph morphism that is injective on star
sets. Subsequently, Margolis and Meakin \cite{MM1} showed how the
theory of {\em inverse monoids} may be used to classify immersions
between connected graphs. These results have been extended by Meakin
and Szak\'acs \cite{MeSz}, \cite{MeSz2} to classify immersions
between higher dimensional cell complexes.

In the present paper we extend the ideas of \cite{MM1} to show how
inverse monoids may be used to study deck transformations of
immersions between connected graphs. By a deck transformation we
mean a graph automorphism that respects the immersion.

In Section 2 of the paper we introduce the  terminology needed to
describe covers and immersions of graphs. We then summarize some of
the classical algebraic and topological ideas involving the
classification of covers and the theory of deck transformations of
connected covers of graphs.

Section 3 summarizes  some of the basic theory of inverse monoids
that will be needed subsequently. We then describe the use of closed
inverse submonoids of free inverse monoids to classify immersions
between connected graphs. We prove an apparently new result
constructing the group of right $\omega$-cosets of a closed inverse
submonoid of an inverse monoid in its normalizer.

In Section 4 we provide a calculation of the group of deck
transformations of a connected immersion between graphs. If $H$ is
the closed inverse submonoid of the free inverse monoid that is used
to construct the immersion, then the group of deck transformations
of the immersion is the group of right $\omega$-cosets of $H$ in its
normalizer (Theorem \ref{deck}).

In Section 5 we describe how the results of earlier sections of the
paper specialize in the case that the graph immersion is actually a
cover of connected graphs. In Section 6 we make an
 observation relating  graph immersions to actions of groups on graphs.
In Section 7 we prove that an immersion between graphs may be
extended to a covering map between graphs in such a  way that deck
transformations of the immersion are restrictions of deck
transformations of the cover.

\section
{Covers and immersions of graphs}

By a {\it graph} ${\Gamma} = (\Gamma^0,\Gamma^1)$ we  mean a graph
in the sense of Serre \cite{Serre}. Here $\Gamma^0$ is the set of
vertices and $\Gamma^1$ is the set of edges of $\Gamma$. Thus every
directed edge $e : v \rightarrow w$ comes equipped with an inverse
edge $e^{-1} : w \rightarrow v$ such that $(e^{-1})^{-1} = e$ and
$e^{-1} \neq e$. The initial vertex of $e$ is denoted by
${\alpha}(e)$ and the terminal vertex of $e$ is denoted by
${\omega}(e)$: thus $\alpha(e^{-1}) = \omega(e)$ and $\omega(e^{-1})
= \alpha(e)$. For each edge $e$ we designate one of the edges  in
the set $\{e,e^{-1}\}$ as being {\em positively oriented}, and its
inverse edge  as being {\em negatively oriented}. We normally only
indicate the positively oriented edges in a  sketch of a graph. A
{\em path} in the graph $\Gamma$ is a finite string $p =
e_1e_2...e_n$ where $\omega(e_i) = \alpha(e_{i+1})$ for $i =
1,...,n-1$: here the edges $e_i$ may be either positively or
negatively oriented. We denote the initial vertex of $p$ by
$\alpha(p)$ and the terminal vertex by $\omega(p)$: that is, if $p =
e_1e_2...e_n$, then $\alpha(p) = \alpha(e_1)$ and $\omega(p) =
\omega(e_n)$.  The {\em inverse} of the path $p = e_1e_2...e_n$ is
the path $p^{-1} = e_n^{-1}...e_2^{-1}e_1^{-1}$. The path $p$ is a
{\em circuit} if $\alpha(p) = \omega(p)$. A {\em tree} is a
connected graph in which every circuit $e_1e_2..e_n$ contains a
subpath of the form $ee^{-1}$ for some edge $e$. Thus the Cayley
graph $\Gamma(X)$ of the free group $FG(X)$ with respect to a set
$X$ of free generators is a tree.

The {\em free category} on a graph $\Gamma$ is the category
$FC(\Gamma)$ whose objects are the vertices of $\Gamma$ and whose
morphisms are the paths in $\Gamma$. The product $p.q$ of paths $p$
and $q$ is defined in $FC(\Gamma)$ if and only if $\omega(p) =
\alpha(q)$ and in that case $p.q = pq$, the concatenation of the
path $p$ followed by the path $q$. We say that a path $p_1$ is an
{\em initial segment} of a path $p$ if there is a path $p_2$ such
that $p = p_1p_2$ and in this case $p_2$ is a {\em terminal segment}
of $p$.

A {\em morphism} from the graph $\Gamma$ to the graph $\Gamma'$  is
a pair of functions $f : \Gamma \rightarrow \Gamma'$ that takes
vertices to vertices and edges to edges, and preserves incidence and
orientation of edges. (Here we abuse notation slightly by using the
same symbol $f$ to denote the corresponding function that takes
vertices to vertices and the function that takes edges to edges.) If
$v \in V({\Gamma})$, let $star({\Gamma},v) = \{e \in E({\Gamma}) :
{\alpha}(e) = v\}$. A morphism $f : {\Gamma} \rightarrow {\Gamma}'$
induces a map $f_{v} : star({\Gamma},v) \rightarrow
star({\Gamma}',f(v))$ between star sets in the obvious way.
Following Stallings \cite{Stall}, we say that a graph morphism $f$
is a {\it cover} if each $f_{v}$ is a bijection and that $f$ is an
{\it immersion} if each $f_{v}$ is an injection.

In his paper \cite{Stall}, Stallings made use of immersions between
graphs to study subgroups of free groups. Subsequently, Stallings
foldings and Stallings graphs have been used extensively to study
subgroups of free groups. See for example the survey paper by
Kapovich and Myasnikov \cite{KM} or the paper by Birget, Margolis,
Meakin and Weil \cite{BMMW} for just some of the relevant
literature.

It is clear from the definition of a graph that it is possible to
label the edges of a graph with labels of positively oriented edges
coming from some set $X $ so that no two positively oriented edges
with the same initial or terminal vertex are assigned the same
label. The labeling of positively oriented edges may be extended to
a labeling of all edges in the graph in such a way that if $e$ is a
positively oriented edge labeled by $x \in X$ then $e^{-1}$ is
labeled by $x^{-1} \in X^{-1}$, where $X^{-1}$ is a set disjoint
from $X$ and in one-one correspondence with $X$ via the map $x
\rightarrow x^{-1}$. We denote the label on an edge $e$ by
$\ell(e)$. Thus $\ell(e) \in X$ if $e$ is a positively oriented edge
and in general $\ell(e) \in X \cup X^{-1}$. The labeling of edges in
$\Gamma$ extends to a labeling of paths in $\Gamma$ in the obvious
way via $\ell(pq) = \ell(p)\ell(q)$ if $pq$ is defined.
  Thus if $p$ is a path in $\Gamma$ then
$\ell(p) \in (X \cup X^{-1})^*$.

For example, the Cayley graph ${\Gamma}(G,X)$ of a group $G$
relative to a set $X$ of generators is obviously labeled over $X
\cup X^{-1}$: its vertices are the elements of $G$ and there is a
directed edge labeled by $x$ from $g$ to $gx$ for each $g \in G$ and
each $x \in X \cup X^{-1}$. The bouquet of circles $B_X$ has one
vertex and one positively oriented edge labeled by $x$ for each $x
\in X$: of course $B_X$ also has negatively oriented edges labeled
by elements $x^{-1} \in X^{-1}$ for each $x \in X$.

If $\Gamma$ is labeled over $X \cup X^{-1}$ as above, then the
associated natural map $f_{\Gamma}$ from ${\Gamma}$ to $ B_{X}$ that
preserves edge labeling is a graph immersion. If $f : \Gamma
\rightarrow \Gamma'$ is a graph immersion and the edges of $\Gamma'$
are labeled over $X \cup X^{-1}$ as above, then this labeling
induces a labeling of the edges of $\Gamma$ in a natural way so that
$f$ preserves edge labeling and
 $ f_{\Gamma'} \circ f = f_{\Gamma}$.

 While the essential results in this
paper may be formulated without resort to labeling the edges of our
graphs, it is more convenient to do so. Hence we adopt the
convention that {\it all graphs that we will consider in this paper
will be edge labeled as described above and all immersions will
preserve edge labeling}.

A simple example illustrating these ideas is provided in Figure 1:
the natural map from the graph ${\Gamma}_{1}$ to $B_{\{a,b\}}$ is a
cover, while the natural map from ${\Gamma}_{2}$ to $B_{\{a,b\}}$ is
an immersion that is not a cover.

\bigskip

\centerline{\includegraphics[height=2.0in]{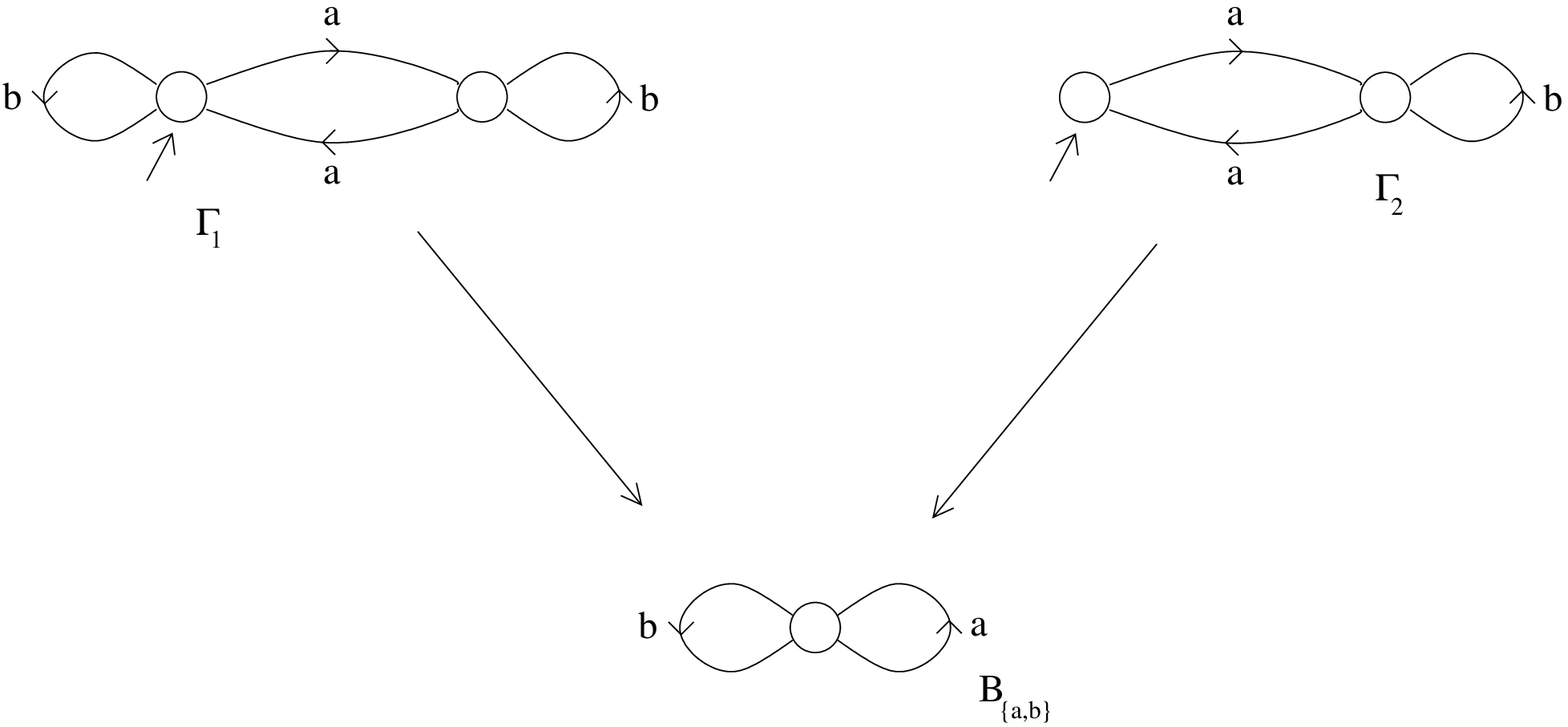} }

\begin{center}

Figure 1

\end{center}


We now list some  straightforward propositions that will be used in
the sequel.

\begin{Prop}
\label{uniquemorph}

An immersion $f: \Gamma \rightarrow \Gamma'$ between connected
edge-labeled graphs is uniquely determined by the image $f(v)$ of
any vertex $v$ in $\Gamma$. More precisely, if $v$ is a vertex of
$\Gamma$ and $v'$ is a vertex of $\Gamma'$, then there is at most
one graph immersion $f : \Gamma \rightarrow \Gamma'$ such that $f(v)
= v'$. Such an immersion exists if and only if, for every path
 $p$ in $\Gamma$ with $\alpha(p) = v$, there is a
path  $p'$  in $\Gamma'$ with $\alpha(p') = v' $ and $\ell(p') =
\ell(p)$ and  such that $p'$ is a  circuit if  $p$ is a circuit.

\end{Prop}
\begin{proof} Suppose that there is an immersion $f$  from
$\Gamma$ to $\Gamma'$ such that $f(v) = v'$. Let $v_1$ be any vertex
of $\Gamma$ and $p$ any path in $\Gamma$ from $v$ to $v_1$. Since
$f$ preserves edge labels, it maps paths in $\Gamma$ to paths in
$\Gamma'$ with the same label, so there must be a  path $p'$ in
$\Gamma'$ with $\alpha(p') = v'$ and $\ell(p') = \ell(p)$.
Furthermore, the path $p'$ is unique since edge labeling is
consistent with an immersion into $B_X$. It follows that we must
have $f(v_1) = \omega(p')$. If $p$ is a circuit, then $\omega(p) =
v$ so $\omega(p') = f(\omega(p)) = f(v) = v'$, and so $p'$ is a
circuit. Also, if $e$ is an edge in $\Gamma$ with $\alpha(e) = v_1$,
then we must  have an edge $e'$ in $\Gamma'$ with $\alpha(e') =
\omega(p')$ and $f(e) = e'$. The uniqueness of such an edge follows
since $\ell(e') = \ell(e)$.   So if there is an immersion from
$\Gamma$ to $\Gamma'$ that maps $v$ to $v'$, there is only one such
immersion.

Suppose conversely that for every path
 $p$ in $\Gamma$ with $\alpha(p) = v$, there is a
path  $p'$  in $\Gamma'$ with $\alpha(p') = v' $ and $\ell(p') =
\ell(p)$ and  such that $p'$ is a  circuit if  $p$ is a circuit. Let
$v_1$ be a vertex in $\Gamma$ and $p$ a path from $v$ to $v_1$ in
$\Gamma$. Then there is a (necessarily unique) path $p'$ in
$\Gamma'$ with $\alpha(p') = v'$ and $\ell(p') = \ell(p)$. Define
$f(v_1) = \omega(p')$. If $p_1$ is another path in $\Gamma$ from $v$
to $v_1$, then as above, there is a unique path $p_1'$ in $\Gamma'$
with $\alpha(p_1') = v'$ and $\ell(p_1') = \ell(p_1)$. Since
$p_1p^{-1}$ is a circuit in $\Gamma$ from $v$ to $v$ it follows by
hypothesis that there is a circuit $q'$ in $\Gamma'$ at $v'$ with
$\ell(q') = \ell(p_1p^{-1}) = \ell(p_1)\ell(p^{-1}) =
\ell(p_1')\ell(p'^{-1}) $. So $p_1'$ is an initial segment of $q'$
and $p'^{-1}$ is a terminal segment of $q'$, and hence $p_1'p'^{-1}$
is a circuit in $\Gamma'$ from $v'$ to $v'$ with $\ell(p_1'p'^{-1})
= \ell(q')$. It follows that $\omega(p_1') = \omega(p')$, so $f$ is
well defined on vertices. If $e$ is an edge in $\Gamma$ with
$\alpha(e) = v_1$, then $pe$ is a path in $\Gamma$ starting at $v$
so there is some path $s'$ in $\Gamma'$ starting at $v'$ with
$\ell(s') = \ell(pe) = \ell(p)\ell(e) = \ell(p')\ell(e)$. By
uniqueness of a path starting at $v'$ with this label and the fact
that $p'$ is a path starting at $v'$ with $\ell(p') = \ell(p)$, it
follows that there must be a (unique) edge $e'$ in $\Gamma'$ with
$\alpha(e') = \omega(p')$ and $\ell(e') = \ell(e)$. Then define
$f(e) = e'$. This is well defined by uniqueness of the edge $e'$. So
$f$ is a well defined morphism from $\Gamma$ to $\Gamma'$ which is
clearly an immersion since it preserves edge labeling.
\end{proof}

An {\em isomorphism} from the edge labeled graph $\Gamma$ onto the
edge labeled graph $\Gamma'$ is a (label-preserving) graph morphism
that is a bijection from vertices of $\Gamma$ to vertices of
$\Gamma'$ and also a bijection from edges of $\Gamma$ to edges of
$\Gamma'$. Such an isomorphism must restrict to isomorphisms between
the connected components of $\Gamma$ and the connected components of
$\Gamma'$ and,  by Proposition \ref{uniquemorph}, it is determined
by the image of any vertex in $\Gamma$ if $\Gamma$ is connected. The
following fact is an easy consequence of Proposition
\ref{uniquemorph}.

\begin{Prop}
\label{graphisomorphisms}

Let $\Gamma$ and $\Gamma'$ be (edge labeled) connected graphs, $v$ a
vertex in $\Gamma$ and $v'$ a vertex in $\Gamma'$. Then a bijective
map $f: \Gamma \rightarrow \Gamma'$ such that $f(v) = v'$ is an
isomorphism from $\Gamma$ onto $\Gamma'$  if and only if there is a
label-preserving bijection $p \leftrightarrow p'$  between the paths
$p$ in $\Gamma$ starting at $v$ and the paths $p'$ in $\Gamma'$
starting at $v'$  such that $p$ is a circuit at $v$ if and only if
$p'$ is a circuit at $v'$.

\end{Prop}

Let $f: \tilde{\Gamma} \rightarrow \Gamma$ be an immersion of
connected graphs, $v$ a vertex in $\Gamma$ and $\tilde{v} \in
f^{-1}(v)$. Then we say that a path $p$ in $\Gamma$ starting at $v$
{\em lifts} to $\tilde{v}$ if there is a path $\tilde{p}$ in
$\tilde{\Gamma}$ starting at $\tilde{v}$ such that $f(\tilde{p}) =
p$. In this case it is clear that the lifted path $\tilde{p}$ is
uniquely determined by $\tilde{v}$ and $p$  since $\ell(\tilde{p}) =
\ell(p)$. We refer to the proposition below as the ``path lifting"
proposition.

\begin{Prop}
\label{pathlifting}

Let $f: \tilde{\Gamma} \rightarrow \Gamma$ be an immersion between
connected graphs. Then $f$ is a cover if and only if, for each
vertex $v$ in $\Gamma$ and each vertex $\tilde{v}$ in $f^{-1}(v)$,
every path $p$ in $\Gamma$ starting at $v$ lifts to a (unique) path
in $\tilde{\Gamma}$ starting at $\tilde{v}$.

\end{Prop}
\begin{proof} If $f$ is a cover, then the fact that every
path $p$ lifts to all preimages of $v$ is  easy to prove and may be
viewed as a very special case of a much more general path lifting
property for covers of topological spaces (see for example
Proposition 1.30 of Hatcher's book \cite{hatch}). Conversely,
suppose that every path in $\Gamma$ starting at $v$ lifts to a path
in $\tilde{\Gamma}$ starting at $\tilde{v}$. It follows that if  $e$
is an edge in $\Gamma$ with $\alpha(e) = v$, then  there is an edge
$\tilde{e}$ in $\tilde{\Gamma}$ with $\alpha(\tilde{e}) = \tilde{v}$
and $f(\tilde{e}) = e$. So the map $f_{\tilde{v}}$ from
$star(\tilde{\Gamma},\tilde{v})$ to
 $star(\Gamma,v)$ is surjective. Since it is also
injective by hypothesis, it is a bijection, and so the immersion $f$
is a cover.
\end{proof}

The path lifting proposition above does not hold for immersions that
are not covers in general, but it is easy to see that maximal
initial segments of paths in $\Gamma$ lift uniquely to paths in
$\tilde{\Gamma}$, as described in the following proposition.

\begin{Prop}
\label{Liftingimmersions}

Let $f: \tilde{\Gamma} \rightarrow \Gamma$ be an immersion between
connected graphs, let $v$ be a vertex of $\Gamma$ and let $p$ be a
path in $\Gamma$ with ${\alpha}(p) = v$. Then for every vertex
$\tilde{v} \in f^{-1}(v)$ there is a unique  (possibly empty)
maximal initial segment $p_1$ of $p$  that lifts to a path at
$\tilde{v}$. Furthermore, the lift of $p_1$ at $\tilde{v}$ is
unique.

\end{Prop}
\begin{proof} Since an immersion is locally injective on
star sets, it is clear that an edge $e$ of $\Gamma$ starting at $v$
lifts to at most one edge $\tilde{e}$ in $\tilde{\Gamma}$ starting
at $\tilde{v}$. The result then follows by an easy inductive
argument.
\end{proof}
We briefly summarize some of the most basic facts linking group
theory and covers of graphs. Recall (see Stallings \cite{Stall})
that two paths $p$ and $q$ in a graph $\Gamma$ are said to be {\em
homotopy equivalent} (written $p \sim q$) if and only if it is
possible to pass from $p$ to $q$ by a finite sequence of insertions
or deletions of  paths of the form $ee^{-1}$ for various edges $e$
of $\Gamma$. Clearly $\alpha(p) = \alpha(q)$ and $\omega(p) =
\omega(q)$ if $p \sim q$. We denote the equivalence class (homotopy
class) of a path $p$ in $\Gamma$ by $[p]$. The {\em fundamental
groupoid }  $\pi_1(\Gamma)$ is a groupoid whose  objects are the
vertices of $\Gamma$ and whose morphisms are the homotopy classes
$[p]$. We regard $[p]$ as a morphism from $\alpha(p)$ to
$\omega(p)$. The multiplication in the groupoid is defined by
$[p][q] = [pq]$ if $\omega(p) = \alpha(q)$ and is undefined
otherwise.

For each vertex $v$ of $\Gamma$ the set $\pi_1(\Gamma,v) = \{[p] :
p$ is a circuit from $v$ to $v$ in $\Gamma \}$ is a group with
respect to the multiplication in $\pi_1(\Gamma)$, called the {\em
fundamental group} of $\Gamma$ based at $v$. The following fact is
classical and can be found in many sources, for example
\cite{Stall}.

\begin{Prop}
\label{fundamentalgroup}

Let $\Gamma$ be a connected graph and $v, w$ vertices of $\Gamma$.
Then

(a) $\pi_1(\Gamma,v)$ is a free group whose rank is the number of
positively oriented edges of $\Gamma$ that are not in a spanning
tree for $\Gamma$.

(b) If $p$ is a path in $\Gamma$ from $v$ to $w$, then the map $[q]
\rightarrow [p][q][p^{-1}]$ for $[q] \in \pi_1(\Gamma,w)$ defines an
isomorphism from $\pi_1(\Gamma,w)$ onto $\pi_1(\Gamma,v)$.

\end{Prop}

It is well known (see for example Hatcher's book \cite{hatch}) that
under suitable conditions on a connected topological space $\mathcal
X$, connected covers of $\mathcal X$ may be classified via conjugacy
classes of subgroups of the fundamental group of $\mathcal X$. The
following version of this result for graph covers may be found  in
many sources, for example, \cite{hatch} or \cite{Stall}.

\begin{Theorem}
\label{coversclass}

(a) Let $f: \tilde{\Gamma} \rightarrow \Gamma$ be a cover of
connected graphs, let $v$ be a vertex of $\Gamma$ and $\tilde{v} \in
f^{-1}(v)$. Then $f$ induces an embedding  of
$\pi_1(\tilde{\Gamma},\tilde{v})$ into $\pi_1(\Gamma,v)$. If
$\tilde{v_1}$ is another vertex in $f^{-1}(v)$, then the groups
$f(\pi_1(\tilde{\Gamma},\tilde{v}))$ and
$f(\pi_1(\tilde{\Gamma},\tilde{v_1}))$ are conjugate subgroups of
$\pi_1(\Gamma,v)$.

(b) Conversely, let $\Gamma$ be a connected graph, $v$ a vertex in
$\Gamma$ and $H \leq \pi_1(\Gamma,v)$.  Then there is a unique (up
to labeled graph isomorphism) connected graph $\tilde{\Gamma}$ and a
unique (up to  equivalence) covering map $f: \tilde{\Gamma}
\rightarrow \Gamma$ and a vertex $\tilde{v} \in \tilde{\Gamma}$ such
that $f(\tilde{v}) = v$ and $f(\pi_1(\tilde{\Gamma},\tilde{v})) =
H$.

\end{Theorem}

If $f: \tilde{\Gamma} \rightarrow \Gamma$ is an immersion of
connected graphs, a labeled graph automorphism $\gamma$ of
$\tilde{\Gamma}$ is called a {\em deck transformation} of
$\tilde{\Gamma}$ if $f = f \circ \gamma$, i.e. $f(\tilde{v}) =
f(\gamma(\tilde{v}))$ for all vertices $\tilde{v}$ in
$\tilde{\Gamma}$. The deck transformations of $\tilde{\Gamma}$ form
a group $G(\tilde{\Gamma})$ with respect to composition of
automorphisms.

A graph  cover $f: \tilde{\Gamma} \rightarrow \Gamma$ is called a
{\em normal} cover if, for every vertex $v$ in $\Gamma$ and every
pair of vertices $\tilde{v}_1, \tilde{v}_2 \in f^{-1}(v)$, there is
a deck transformation that takes $\tilde{v}_1$ to $\tilde{v}_2$:
equivalently (\cite{hatch}, Proposition 1.39), $f$ is a normal cover
if and only if the subgroup $H = f(\pi_1(\tilde{\Gamma},\tilde{v}))$
of $\pi_1(\Gamma,f(v))$  that defines the cover is a normal subgroup
of $\pi_1(\Gamma,f(v))$. The {\em universal cover} of $\Gamma$ is
 the cover $\tilde{\Gamma}$ corresponding to the trivial subgroup of
 $\pi_1(\Gamma,f(v))$: a cover of $\Gamma$ is isomorphic to the
 universal cover if and only if it is a tree.

If $f : \tilde{\Gamma} \rightarrow \Gamma$ is a {\em cover} of
connected graphs, then there is a well known connection between the
group $G(\tilde{\Gamma})$ of deck transformations and the
fundamental group of $\Gamma$. The following result is a special
case of a more general standard result in topology (see for example
\cite{hatch}, Proposition 1.39).

\begin{Theorem}
\label{deckcovers}

Let $f: \tilde{\Gamma} \rightarrow \Gamma$ be a cover of connected
graphs, let $\tilde{v}$ be a vertex of $\tilde{\Gamma}$ with
$f(\tilde{v}) = v$ and let $H$ be the subgroup $H =
f(\pi_1(\tilde{\Gamma},\tilde{v}))$ of $\pi_1(\Gamma,v)$. Then
$G(\tilde{\Gamma}) \cong N(H)/H$, where $N(H)$ is the normalizer of
$H$ in $\pi_1(\Gamma,v)$. In particular, if $\tilde{\Gamma}$ is the
universal cover of $\Gamma$, then $G(\tilde{\Gamma}) \cong
\pi_1(\Gamma,v)$.

\end{Theorem}

While group theory provides a powerful algebraic tool for
classifying and studying {\em covers} of graphs (or topological
spaces in general), it appears that groups do not provide an
adequate algebraic tool to classify {\em immersions} between graphs.
For example let $\Gamma(a)$ denote the Cayley graph of $\mathbb Z =
Gp\langle a : \emptyset \rangle$ with respect to the generating set
$\{a\}$. The vertices of $\Gamma(a)$ may be identified with the
integers and there is a directed edge labeled by $a$ from $n$ to
$n+1$ for each integer $n$. Then $\Gamma(a)$ is the universal cover
of the circle $B_{\{a\}}$. Any connected subgraph of $\Gamma(a)$
immerses into $B_{\{a\}}$ but all such graphs have trivial
fundamental groups, so they cannot be distinguished by subgroups of
$\mathbb Z$. We need a different algebraic tool to classify
immersions and to encode the fact that paths only sometimes lift
under immersions (Proposition \ref{Liftingimmersions}). In
subsequent sections of this paper, we show that the theory of {\em
inverse monoids} provides a useful algebraic tool to study graph
immersions and in particular to obtain analogues of Theorems
\ref{coversclass} and \ref{deckcovers}.

\section{Inverse monoids and their closed inverse submonoids }

An {\em inverse monoid} is a monoid $M$ such that for every element
$a \in M$ there is a {\em unique} element $a^{-1}$ in $M$ such that

\begin{center}
$a = aa^{-1}a \,\,\,\,$ and $\,\,\,\, a^{-1} = a^{-1}aa^{-1}.$
\end{center}

It is clear that the elements $aa^{-1}$ and $a^{-1}a$ of an inverse
monoid are {\em idempotents} of $M$. In general, $aa^{-1} \neq
a^{-1}a$ and neither of these idempotents is necessarily equal to
the identity $1$ of $M$. We denote the set of idempotents of an
inverse monoid $M$ by $E(M)$. An important elementary fact about
inverse monoids is that their {\em idempotents commute}, i.e. $ef =
fe$ for all $e,f \in M$. In fact  inverse monoids may be
 characterized alternatively as  (von Neumann) regular monoids whose
idempotents commute.

Inverse monoids provide an appropriate algebraic tool for studying
{\em partial symmetry} of mathematical objects in much the same way
as groups are used to study symmetry. We refer the reader to the
book of Lawson \cite{Law} for an exposition of this point of view
and for much basic information about inverse monoids.

A standard example of an inverse monoid is the {\em symmetric
inverse monoid} $SIM(X)$ on a set $X$. This is the set of bijections
between subsets of $X$ with respect to the usual composition of
partial maps. The inverse of a bijection $f \in SIM(X)$ with domain
$A$ and range $B$ is the inverse map $f^{-1}$ with domain $B$ and
range $A$: the idempotents of $SIM(X)$ are the identity maps on
subsets of $X$ (including the empty map $0$ with domain the empty
subset). The analogue for inverse monoids of Cayley's theorem for
groups is the {\em Wagner-Preston Theorem}, namely every inverse
monoid embeds in a suitable symmetric inverse monoid (see \cite{Law}
for a proof of this theorem).

The {\em natural partial order} on an inverse monoid $M$ is defined
by $a \leq b$ iff $a = eb$ for some idempotent $e \in E(M)$
(equivalently $a = bf$ for some idempotent $f \in E(M)$ or
equivalently $a = aa^{-1}b$ or equivalently $a = ba^{-1}a$). This
extends the natural partial order on $E(M)$ defined by  $e \leq f$
iff $e = ef = fe$. With respect to this partial order, $E(M)$ forms
a {\em lower semilattice} with meet operation $e \wedge f = ef$ for
all $e,f \in E(M)$. For the symmetric inverse monoid $SIM(X)$, the
natural partial order corresponds to restriction of a partial
one-one map to a subset of its domain. The semilattice of
idempotents of $SIM(X)$ is of course the Boolean lattice of subsets
of $X$ with respect to inclusion.


For each subset $N$ of an inverse monoid $M$, we denote by
$N^{\omega}$ the set of all elements $m \in M$ such that $m \geq n$
for some $n \in N$.  The subset $N$ of $M$ is called {\em closed} if
$N = N^{\omega}$. Closed inverse submonoids of an inverse monoid $M$
arise naturally in the representation theory of $M$ by partial
injections on a set developed by Schein
 \cite{schein2}. An inverse monoid $M$ acts  by
 injective partial functions on a set $Q$ if there is a homomorphism
 from $M$ to $SIM(Q)$.  Denote by $qm$ the image of $q$ under the
 action of $m$ if $q$ is in the domain of the action by $m$. (Here
 we are considering $M$ as acting on the right on $Q$.)

If an inverse monoid $M$ acts on $Q$ by injective partial functions,
then for every $q \in
 Q, \, Stab(q) = \{m \in M : qm = q\}$ is a closed inverse submonoid
 of $M$. Conversely, given a  closed inverse submonoid $H$ of $M$, we
 can construct a transitive representation of $M$ as follows. A
 subset of $M$ of the form $(Hm)^{\omega}$ where $mm^{-1} \in H$ is
 called a {\em right ${\omega}$-coset} of $H$. Let $X_H$ denote the
 set of right $\omega$-cosets of $H$. If $(Hn)^{\omega} \in X_H$ and
 $m \in M$, define an action by
  $(Hn)^{\omega}  m = (Hnm)^{\omega}$ if
 $(Hnm)^{\omega} \in X_H$ and
 undefined otherwise. This defines a transitive action of $M$ (on the
 right) on
 $X_H$. Conversely, if $M$ acts transitively on $Q$, then this action
 is equivalent in the obvious sense to the action of $M$ on the
 right $\omega$-cosets of $Stab(q)$ in $M$ for any $q \in Q$. See
 \cite{schein2} or \cite{Pet}  for details. Dually, {\em left
 $\omega$-cosets} of $H$ in $M$ (sets of the form $(mH)^{\omega}$ where $m^{-1}m \in H$)
 arise in connection with left actions of
 $M$ by partial one-one maps on some set.

 If $M$ is generated as
 an inverse monoid by a set $X$ and $H$ is a closed inverse
 submonoid of $M$ then the graph of
 right $\omega$-cosets of $H$ in $M$  is the graph with
 vertices the set $X_H$ of right $\omega$-cosets of $H$ in $M$ and with an edge
 labeled by $x \in X \cup X^{-1}$ from $(Hm)^{\omega}$ to
 $(Hmx)^{\omega}$ if $mm^{-1}, mxx^{-1}m^{-1} \in H$.

It is well known  that free inverse monoids exist. We will denote
the free inverse monoid on a set $X$ by $FIM(X)$. The structure of
free inverse monoids on one generator was determined by Gluskin
\cite{Gluskin}. The structure of free inverse monoids in general was
determined much later independently by Scheiblich \cite{Sch} and
Munn \cite{Munn2}. Scheiblich's description  for elements of
$FIM(X)$ is in terms of rooted Schreier subsets of the free group
$FG(X)$, while Munn's description is in terms of birooted
edge-labeled trees. Scheiblich's description provides an important
example of a McAlister triple, in the spirit of the McAlister
$P$-theorem \cite{Mc2}, while Munn's description lends itself most
directly to a solution to the word problem for $FIM(X)$. It is not
difficult to see the equivalence of the two descriptions.
A variation on Scheiblich's approach is provided in Lawson's book
\cite{Law}. The version below is a slight variation on Munn's
approach, the essential difference being that for some purposes it
is somewhat more convenient to regard Munn's birooted trees as
subtrees of the Cayley graph of the free group $FG(X)$, with the
initial root identified with the vertex $1$ in the Cayley graph.

Denote by ${\Gamma}(X)$ the Cayley graph of the free group $FG(X)$
with respect to the usual presentation, $FG(X) = Gp \langle X :
\emptyset \rangle$. Thus ${\Gamma}(X)$ is an infinite tree whose
vertices correspond to the elements of $FG(X)$ (in reduced form) and
with a directed edge labeled by $x \in X$ from $g$ to $gx$ (and an
inverse edge labeled by $x^{-1}$ from $gx$ to $g$). For each word $w
\in (X \cup X^{-1})^*$, denote by $MT(w)$ the finite subtree of the
tree ${\Gamma}(X)$ obtained by reading the word $w$ as the label of
a path in ${\Gamma}(X)$, starting at $1$. Thus, for example, if $w =
aa^{-1}bb^{-1}ba^{-1}abb^{-1}$, then $MT(w)$ is the tree pictured in
Figure 2.

\bigskip

\centerline{\includegraphics[height=1.3in]{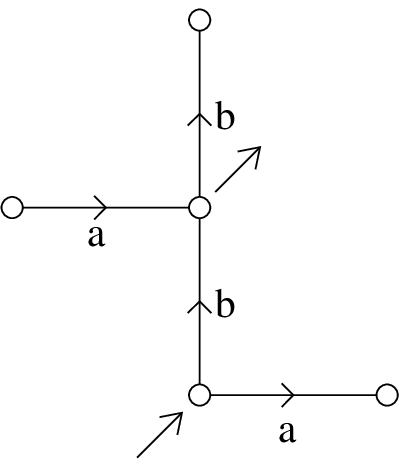}}

\medskip

\begin{center}
Figure 2
\end{center}

\bigskip

One may view $MT(w)$ as a birooted tree, with initial root $1$ and
terminal root $r(w)$, the reduced form of the word $w$ in the usual
group-theoretic sense. Munn's solution \cite{Munn2} to the word
problem in $FIM(X)$ may be stated in the following form.

\begin{Theorem}
 If $u,v \in (X \cup X^{-1})^*$, then $u = v$ in $FIM(X)$ iff
$MT(u) = MT(v)$ and $r(u) = r(v)$.

\end{Theorem}

Thus elements of $FIM(X)$ may be viewed as pairs $(MT(w),r(w))$ (or
as birooted edge-labeled trees, which was the way that Munn
described his results).  Multiplication in $FIM(X)$ is performed as
follows. If $u,v \in (X \cup X^{-1})^*$, then $MT(uv) = MT(u) \cup
r(u).MT(v)$ (just translate $MT(v)$ so that its initial root
coincides with the terminal root of $MT(u)$ and take the union of
$MT(u)$ and the translated copy of $MT(v)$: the terminal root is of
course $r(uv)$). The idempotents of $FIM(X)$ are the {\em Dyck
words}, i.e. the words $w$ whose reduced form is $1$: two Dyck words
represent the same idempotent of $FIM(X)$ if and only if they have
the same Munn tree. Munn's approach has been greatly extended by
Stephen \cite{Ste2} to a general theory of presentations of inverse
monoids by generators and relations.

 In their paper \cite{MM1}, Margolis and Meakin studied closed
 inverse submonoids of free inverse monoids and showed how they
 could be used to classify immersions between connected graphs. In particular, they showed that
 closed inverse submonoids of free inverse monoids have surprisingly nice finiteness properties,
 and they may be constructed from free actions of groups on trees.
 A closed inverse submonoid of a free inverse monoid is  not necessarily a free
 inverse monoid, but it admits an idempotent pure morphism onto a free inverse monoid.  We
 briefly recall some of the results from the paper \cite{MM1} that are
 relevant for our current purposes.

 Recall that an {\em inverse category} $\mathcal C$ is a category with the property that for
 each morphism
 $p$ there is a unique morphism $p^{-1}$ such that $ p = pp^{-1}p$ and
 $p^{-1} = p^{-1}pp^{-1}$.
 The {\em loop monoid} of such a category at the vertex (object)
 $v$ is the set $L({\mathcal C},v) = \{p: p$ is a morphism in $\mathcal C$ from $v$
 to $v\}$. Then $L({\mathcal C},v)$ is an inverse monoid.
Let $\Gamma$ be a (connected) graph.  Recall from \cite{MM1} that
 the {\em free inverse category} $FIC(\Gamma)$ on $\Gamma$ is the
 quotient of the free category $FC(\Gamma)$ on $\Gamma$ by the
 category congruence $\sim_i$ induced by all relations of the form
 $p = pp^{-1}p$ and $pp^{-1}qq^{-1} = qq^{-1}pp^{-1}$ if $\alpha(p)
 = \alpha(q)$ for paths $p,q$ in $\Gamma$. Denote the $\sim_i$-class
 of $p$ by $[[p]]$. Then $FIC(\Gamma)$ is an  inverse category. Of
 course if $\Gamma = B_X$ then
 $FIC(\Gamma) = FIM(X)$.

 The {\em loop monoid} of $FIC(\Gamma)$ at the vertex $v$ of
 $\Gamma$ is

 \begin{center}

 $L(\Gamma,v) = \{[[p]]: p$ is a circuit in $\Gamma$
 based at $v\}$.

 \end{center}

If $\Gamma$ is labeled over $X \cup X^{-1}$
 consistent with an immersion into $B_X$  it follows that
 if $p \, \sim_i \, q$
 for paths $p,q$ in $\Gamma$, then $\ell(p)$ and $\ell(q)$ are equal
 in $FIM(X)$.
So we will {\em view labels of paths in $\Gamma$ as
 elements of $FIM(X)$} throughout the sequel.
We note that $FIM(X)$ acts on
 $\Gamma^0$ in a natural way (namely $w \in FIM(X)$ acts on $v_1$ and takes $v_1$ to
 $v_2$ if there is a path labeled by $w$ from $v_1$ to $v_2$). Then $L(\Gamma,v)$ is the stabilizer
 of $v$ under this action, so {\em each loop monoid is a closed
 inverse submonoid of $FIM(X)$}. See Proposition 4.3 of \cite{MM1}
 for details.

 Recall that two closed inverse submonoids $H$ and $K$ of an inverse
 monoid $M$ are said to be {\em conjugate} (written $H \approx K$)
 if there exists $m \in M$ such that $mHm^{-1} \subseteq K$ and
 $m^{-1}Km \subseteq H$. Conjugation is an equivalence relation on
 the set of closed inverse submonoids of $M$.
 We caution however that, unlike the situation in group theory,
 conjugate closed inverse submonoid of inverse monoids are not
 necessarily isomorphic. For example, the subsets
 $\{1,aa^{-1},a^2a^{-2}\}$ and $\{1,aa^{-1},a^{-1}a,aa^{-2}a\}$ of
 $FIM(\{a\})$ are conjugate closed inverse submonoids of
 $FIM(\{a\})$ that are clearly not isomorphic.

 Immersions of connected graphs over $B_X$ are classified via
 conjugacy classes of closed inverse submonoids of $FIM(X)$ as
 indicated in the following theorem (Theorem 4.4 of \cite{MM1}).

\begin{Theorem}
\label{classimmb}

Let $\Gamma$ be a connected graph with edges labeled over $X \cup
X^{-1}$ consistent with an immersion into $B_X$. Then each loop
monoid is a closed inverse submonoid of $FIM(X)$ and the set of all
loop monoids $L(\Gamma,v)$ for $v$ a vertex of $\Gamma$ is a
conjugacy class of the set of closed inverse submonoids of $FIM(X)$.
Conversely, if $H$ is any closed inverse submonoid of a free inverse
monoid $FIM(X)$ then there is some graph $\Gamma$ and an immersion
$f: \Gamma \rightarrow B_X$ such that $H$ is a loop monoid of
$FIC(\Gamma)$: furthermore, $\Gamma$ is unique (up to graph
isomorphism) and $f$ is unique (up to equivalence).

\end{Theorem}

The results of this theorem can be extended somewhat to obtain a
classification of immersions over arbitrary graphs, as indicated in
the following theorem (Theorem 4.5 of \cite{MM1}). This theorem may
be viewed as the analogue for graph immersions of the classification
theorem for graph covers (Theorem \ref{coversclass} above).

\begin{Theorem}
\label{classimm}

Let $f: \Delta \rightarrow \Gamma$ be an immersion of connected
graphs where $\Delta$ and $\Gamma$ are edge labeled over $X \cup
X^{-1}$ consistent with immersions into $B_X$. If $v$ is a vertex of
$\Gamma$ and $v_1 \in f^{-1}(v)$ then $f$ induces an embedding of
$L(\Delta,v_1)$ into $L(\Gamma,v)$. Conversely, let $\Gamma$ be a
graph edge-labeled over $X \cup X^{-1}$ as usual and let $H$ be a
closed inverse submonoid of $FIM(X)$ so that $H \subseteq
L(\Gamma,v)$ for some vertex $v$ in $\Gamma$. Then there is a graph
$\Delta$, an immersion $f: \Delta \rightarrow \Gamma$ and a vertex
$v_1$ in $\Delta$ such that $f(v_1) = v$ and $f(L(\Delta,v_1)) = H$.
Furthermore, $\Delta$ is unique (up to graph isomorphism) and $f$ is
unique (up to equivalence). If $H,K$ are two closed inverse
submonoids of $FIM(X)$ with $H,K \subseteq L(\Gamma,v)$, then the
corresponding immersions $f: \Delta \rightarrow \Gamma$ and $g:
\Delta' \rightarrow \Gamma$ are equivalent iff $H \approx K$ in
$FIM(X)$.

\end{Theorem}

\noindent {\bf Remark} From the proof of Theorem \ref{classimm} in
\cite{MM1}, it follows that the graph $\Delta$ constructed in this
theorem is the graph  of right $\omega$-cosets of $H$ in $FIM(X)$.

\medskip

Theorems \ref{classimmb} and \ref{classimm} have been  extended to
classify immersions between $2$-dimensional $CW$-complexes and more
generally between finite dimensional $\Delta$-complexes in the
papers by Meakin and Szak\'acs \cite{MeSz} and \cite{MeSz2}. In
these more general cases it is necessary to construct closed inverse
submonoids of certain inverse monoids presented by generators and
relations associated with the complexes.

We close this section with a result about normalizers of closed
inverse submonoids of inverse monoids. Part (a) of  Proposition
\ref{cosets} below was observed in a paper by Lawson, Margolis and
Steinberg (\cite{LMS}, Lemma 2.10),  but as far as we know the other
parts of the proposition are new. This result will be needed in the
construction of the group of deck transformations of a graph
immersion later in this paper.

Let $H$ be a closed inverse submonoid
 of an inverse monoid $M$.  Define the {\it normalizer} $N(H)$ of $H$
 in $M$ to be the set

$$N(H) = \{a \in M : aHa^{-1}, \, a^{-1}Ha \subseteq H\}$$

 \begin{Prop} \label{cosets}

 Let $H$ be a closed inverse submonoid of an inverse monoid $M$. Then
\begin{enumerate}[label=(\alph*)]
 \item $N(H)$ is a closed inverse submonoid of $M$ and $H$ is a full
 inverse submonoid of $N(H)$ (i.e. $H$ contains all of the idempotents of $N(H)$);

 \item If $a \in N(H)$ then the right ${\omega}$-coset $(Ha)^{\omega}$
 and the left ${\omega}$-coset $(aH)^{\omega}$ both exist and
 $(Ha)^{\omega} = (aH)^{\omega}$;


\item The relation $\rho_H$ on $N(H)$ defined by $a \, \rho_H \, c$
iff $(Ha)^{\omega} = (Hc)^{\omega}$ is a congruence on $N(H)$ and
$N(H)/{\rho_H}$ is a group with operation $(Ha)^{\omega} .
(Hb)^{\omega} = (Hab)^{\omega}$.

\end{enumerate}
 \end{Prop}

 \begin{proof}

 (a) It is clear that $a^{-1} \in N(H)$ if $a
 \in N(H)$ and also that $N(H)$ contains the identity of $M$. Also if
 $a,b \in N(H)$ then $(ab)^{-1}H(ab) = b^{-1}a^{-1}Hab \subseteq
 b^{-1}Hb \subseteq H$ and similarly $(ab)H(ab)^{-1} \subseteq H$.
Thus $ab \in N(H)$, so $N(H)$ is an inverse submonoid of $FIM(X)$.
 If $a \in N(H)$ and $b \geq a$ then $a = be$ for some idempotent $e
 \in M$. Hence if $h \in H$ then $beheb^{-1} = aha^{-1} \in H$.
 Hence
 $bhb^{-1} \in H$ since $H$ is closed and $bhb^{-1} \geq beheb^{-1}$.
 Thus $bHb^{-1} \subseteq H$: similarly $b^{-1}Hb \subseteq H$.
 Hence $b \in N(H)$ and so $N(H)$ is a closed inverse submonoid of
 $M$. It is clear that $H \subseteq N(H)$. If $e$ is an idempotent of
 $N(H)$, then $e = e1e^{-1} \in H$ since $1 \in H$ so $H$
 is a full inverse submonoid of $N(H)$.

 (b) If $a \in N(H)$ then $aa^{-1}, \, a^{-1}a \in H$ since $H$ has an identity, so both left
 and right $\omega$-cosets exist. Note  that if $a \in N(H)$ then
 $aHa^{-1} \subseteq H$ and so $a^{-1}aHa^{-1}a \subseteq a^{-1}Ha
 \subseteq H$ so via conjugation by $a$ we get $aHa^{-1} = aa^{-1}a H a^{-1}aa^{-1} \subseteq
 aa^{-1}Haa^{-1} \subseteq aHa^{-1}$. It follows that $aHa^{-1} =
 aa^{-1}Haa^{-1}$ and so $aHa^{-1}a = aa^{-1}Ha$.

 Let $x \in (Ha)^{\omega}$, so $x \geq ha$, for some $h \in H$. Hence $x
 \geq aa^{-1}ha = ah_1a^{-1}a$ for some $h_1 \in H$ by the
 observation above. Since $h_2 = h_1a^{-1}a \in H$, and $x \geq ah_2$
  it follows that $x \in (aH)^{\omega}$. Hence $(Ha)^{\omega}
 \subseteq (aH)^{\omega}$. A similar argument shows the converse
 inclusion.

(c) If $a \in N(H)$ then by part (b), $(Ha)^{\omega}$ exists and
clearly $a \in (Ha)^{\omega}$. This shows that $\rho_H$ is a
reflexive relation on $N(H)$. Since $\rho_H$ is obviously symmetric
and transitive, it is an equivalence relation on $N(H)$.   It is
well known that \omcos{a}=\omcos{c} if and only if $ac^{-1}\in
 H$ (see, for example, \cite{Pet}, Lemma IV.4.5.).
 Suppose $a \, \rho_H \, c$ and $b \, \rho_H \, d$, so
 that \omcos{a}$=$\omcos{c} and \omcos{b}=\omcos{d}. Then
 $ac^{-1}, bd^{-1}\in H$. We have $abd^{-1}c^{-1} \in aHc^{-1}$, so
 $abd^{-1}c^{-1} = ahc^{-1}$, for some $h \in H$. But $ah \in
 (aH)^\omega = (Ha)^\omega$ since $a \in N(H)$, so $ah \geq h_1a$ for
 some $h_1 \in H$. Hence $abd^{-1}c^{-1} \geq h_1ac^{-1} \in H$ and
 so $abd^{-1}c^{-1} \in H$.  Hence \omcos{ab} = \omcos{cd}, i.e. $\rho_H$ is a congruence on $N(H)$.

 The operation of congruence classes in $N(H)/{\rho_H}$ is of course $(Ha)^{\omega}
 .
 (Hb)^{\omega} = (Hab)^{\omega}$.
 Clearly $H$ = \omcos{1} is a right $\omega$-coset. It follows that,
 by definition, \omcos{1}  .   \omcos{a} = \omcos{a} = \omcos{a} .
 \omcos{1}, so $H$ acts as an identity. Also, if $a \in N(H)$, then $a^{-1} \in
 N(H)$ and
\omcos{a} . \omcos{a^{-1}} = \omcos{aa^{-1}} = $H$ since $aa^{-1}
\in H$. Similarly \omcos{a^{-1}} . \omcos{a} = $H$, so
$N(H)/{\rho_H}$ is a group.
\end{proof}

We denote the group $N(H)/{\rho_H}$ by $N(H)/H$ and refer to it as
the {\em group of $\omega$-cosets of $H$ in $N(H)$}.

\section{ The group of deck transformations}

Throughout this section $f : \tilde{\Gamma} \rightarrow \Gamma$ is
an immersion of connected graphs,  $v_0$ is a fixed basepoint in
$\Gamma$, $\tilde{v}_0 \in f^{-1}(v_0)$ is a fixed basepoint in
$\tilde{\Gamma}$ and $H = f(L(\tilde{\Gamma},\tilde{v}_0))$, a
closed inverse submonoid of $L(\Gamma,v_0)$.

A {\em partial isomorphism} of $\tilde{\Gamma}$ is a (labeled graph)
isomorphism $\phi: \Delta_1 \rightarrow \Delta_2$ between subgraphs
$\Delta_1$ and $\Delta_2$ of $\tilde{\Gamma}$ that respects the
immersion, that is $f(\tilde{v}) = f(\phi(\tilde{v}))$ for all
vertices $\tilde{v}$ in $\Delta_1$: it follows that $f(\tilde{e}) =
f(\phi(\tilde{e}))$ for all edges $\tilde{e}$ in $\Delta_1$. It is
convenient to also consider the empty map from $\emptyset$ to
$\emptyset$ as a partial isomorphism: this is denoted by $0$. We
denote the set of partial isomorphisms of $\tilde{\Gamma}$ by
$PI(\tilde{\Gamma})$. Partial isomorphisms may be composed in the
usual way. Namely, if $\phi_1: \Delta_1 \rightarrow \Delta_2$ and
$\phi_2 : \Delta_3 \rightarrow \Delta_4$ are partial isomorphisms,
then $\phi_2 \circ \phi_1$ is the corresponding partial isomorphism
from $\phi_1^{-1}(\Delta_2 \cap \Delta_3)$ to $\phi_2(\Delta_2 \cap
\Delta_3)$ defined by $(\phi_2 \circ \phi_1)(\tilde{e}) =
\phi_2(\phi_1(\tilde{e}))$ for all vertices and edges $\tilde{e} \in
\phi_1^{-1}(\Delta_2 \cap \Delta_3)$.

\begin{Prop}
\label{partialiso}

With respect to the multiplication above, $PI(\tilde{\Gamma})$ is an
inverse monoid. The idempotents of $PI(\tilde{\Gamma})$ are the
identity automorphisms on subgraphs of $\tilde{\Gamma}$ and the
corresponding maximal subgroups are isomorphic to the group of deck
transformations of the subgraph. In particular the group of units of
$PI(\tilde{\Gamma})$ is  the group  $G(\tilde{\Gamma})$ of deck
transformations of $\tilde{\Gamma}$. The natural partial order on
$\tilde{\Gamma}$ is defined by $\phi_1 \leq \phi_2$ if the domain of
$\phi_1$ is a subgraph of the domain of $\phi_2$ and $\phi_1$ is the
restriction of $\phi_2$ to the domain of $\phi_1$.

\end{Prop}

\begin{proof}
The proof that $PI(\tilde{\Gamma})$ is an inverse semigroup that has
the stated properties is a standard routine argument similar to the
proof of the corresponding properties for  $SIM(X)$. For example,
the identity of $PI(\tilde{\Gamma})$ is clearly the identity
automorphism of $\tilde{\Gamma}$. A partial isomorphism $\phi:
\Delta_1 \rightarrow \Delta_2$ is in the group of units of
$PI(\tilde{\Gamma})$ if and only if $\Delta_1 = \Delta_2 =
\tilde{\Gamma}$, so it follows that the group of units is
$G(\tilde{\Gamma})$.
\end{proof}

If $\Delta$ is a connected subgraph of $\tilde{\Gamma}$ and
$\tilde{v}$ is a vertex of $\Delta$, then we define
$H(\Delta,\tilde{v}) = f(L(\Delta,\tilde{v}))$. By Theorem
\ref{classimm} this is a closed inverse submonoid of $FIM(X)$ that
is contained in $L(\Gamma,v)$ where $v = f(\tilde{v})$. In
particular, if $\tilde{v}_0 \in \Delta$, then $H(\Delta,\tilde{v}_0)
\leq H = H(\tilde{\Gamma},\tilde{v}_0)$.

\begin{Lemma} \label{iso}

Let $\Delta_1$ and $\Delta_2$ be connected subgraphs of
$\tilde{\Gamma}$ containing vertices $\tilde{v}_1$ and $\tilde{v}_2$
respectively. Then there is a  partial isomorphism from $\Delta_1$
onto $\Delta_2$ that maps $\tilde{v}_1$ onto $\tilde{v}_2$ if and
only if $H(\Delta_1,\tilde{v}_1) = H(\Delta_2,\tilde{v}_2)$.

\end{Lemma}

\begin{proof} By Proposition \ref{graphisomorphisms} a
 partial isomorphism $\phi$ from $\Delta_1$
onto $\Delta_2$ that maps $\tilde{v}_1$ onto $\tilde{v}_2$ induces a
label-preserving bijection from paths in $\Delta_1$ that start at
$\tilde{v}_1$ to paths in $\Delta_2$ that start at $\tilde{v}_2$ and
 maps circuits to circuits. This bijection also  maps
$\sim_i$- equivalent paths to $\sim_i$- equivalent paths, so it
induces an isomorphism from $L(\Delta_1,\tilde{v}_1)$ onto
$L(\Delta_2,\tilde{v}_2)$. Also $f(\tilde{v}_1) = f(\tilde{v}_2)$.
Denote this vertex of $\Gamma$ by $v$.  Since
$L(\Delta_i,\tilde{v}_i) \cong H(\Delta_i,\tilde{v}_i)$ by Theorem
\ref{classimm}, this means that there is a labeled graph isomorphism
from $H(\Delta_1,\tilde{v}_1)$ onto $H(\Delta_2,\tilde{v}_2)$ that
fixes $v$. But this implies that $H(\Delta_1,\tilde{v}_1) =
H(\Delta_2,\tilde{v}_2)$ since the labeling of edges in $\Gamma$ is
consistent with an immersion into $B_X$.

Conversely, suppose that $H(\Delta_1,\tilde{v}_1) =
H(\Delta_2,\tilde{v}_2)$. Let $\tilde{q_1}$ be a path in $\Delta_1$
starting at $\tilde{v}_1$ and let $q$ be its projection into
$\Gamma$. Then $\tilde{q_1}\tilde{q}_1^{-1} \in
L(\Delta_1,\tilde{v}_1)$, so $qq^{-1} \in H(\Delta_1,\tilde{v}_1) =
H(\Delta_2,\tilde{v}_2)$, so $qq^{-1}$ lifts to a path
$\tilde{q}_2\tilde{q}_2^{-1}$ with $\tilde{q}_2\tilde{q}_2^{-1} \in
L(\Delta_2,\tilde{v}_2)$. Hence $q$ lifts to the path $\tilde{q}_2$
in $\Delta_2$ starting at $\tilde{v}_2$. Clearly
${\ell}(\tilde{q}_1) = {\ell}(\tilde{q}_2) = {\ell}(q)$. Dually, if
$\tilde{q}_2$ is a path in $\Delta_2$ starting at $\tilde{v}_2$ then
there is a corresponding path $\tilde{q}_1$ in $\Delta_1$ starting
at $\tilde{q}_1$. The correspondence $\tilde{q}_1 \leftrightarrow
\tilde{q}_2$ is one-to-one by equality of labels on these paths.
Since it also maps circuits to circuits, it induces an isomorphism
from $\Delta_1$ onto $\Delta_2$ that maps $\tilde{v}_1$ onto
$\tilde{v}_2$ by Proposition \ref{graphisomorphisms}. \end{proof}

If $v$ is a vertex in the connected graph $\Gamma$ and $H$ is a
closed inverse submonoid of $L(\Gamma,v)$, then we denote by
$N(H,v)$ the normalizer of $H$ in $L(\Gamma,v)$, that is $N(H,v) =
\{p \in L(\Gamma,v) : pHp^{-1},p^{-1}Hp \subseteq H\}$.

\begin{Lemma} \label{normalizerlift2} Let $f : \tilde{\Gamma}
\rightarrow \Gamma$ be an immersion between connected graphs, $v_0$
a vertex of $\Gamma$, $\tilde{v}_0 \in f^{-1}(v_0)$ and $H =
f(L(\tilde{\Gamma},\tilde{v}_0))$. Then

\begin{enumerate}[label=(\alph*)]

\item If there is a deck transformation that maps $\tilde{v}_0$ to
$\tilde{v}_1$, then any path $\tilde{p}$ from $\tilde{v}_0$ to
$\tilde{v}_1$ projects onto a path $p \in N(H,v_0)$.

\item Each path $p \in N(H,v_0)$ lifts to a path $\tilde{p}$ in
$\tilde{\Gamma}$ starting at $\tilde{v}_0$.

\item Each path $p \in N(H,v_0)$ determines a deck
transformation ${\phi}_{p}$ of $\tilde{\Gamma}$ that maps
$\tilde{v}_0$ to ${\omega}(\tilde{p}) \in f^{-1}(v_0)$. Furthermore,
if $q$ is another element of $N(H,\tilde{v}_0)$, then ${\phi}_{p} =
{\phi}_{q}$ iff ${\omega}(\tilde{p}) = {\omega}(\tilde{q})$.

\end{enumerate}
\end{Lemma}

\begin{proof}

(a) Suppose there is a deck transformation that maps $\tilde{v}_0$
to $\tilde{v}_1$ and $\tilde{p}$ is a path from $\tilde{v}_0$ to
$\tilde{v}_1$. Then
$\tilde{p}L(\tilde{\Gamma},\tilde{v}_1)\tilde{p}^{-1} \subseteq
L(\tilde{\Gamma},\tilde{v}_0)$ and
$\tilde{p}^{-1}L(\tilde{\Gamma},\tilde{v}_0)\tilde{p} \subseteq
L(\tilde{\Gamma},\tilde{v}_1)$. Since
$f(L(\tilde{\Gamma},\tilde{v}_1)) = f(L(\tilde{\Gamma},\tilde{v}_0))
= H$ by Lemma \ref{iso}, this implies that $pHp^{-1} \subseteq H$
and $p^{-1}Hp \subseteq H$, so $p \in N(H,v_0)$.

(b) Let $p \in N(H,v_0)$. Since $pp^{-1} \in H$ and every circuit in
$H$ based at $v_0$ lifts to a circuit in $\tilde{\Gamma}$ based at
$\tilde{v}_0$ it follows that $pp^{-1}$ lifts to a circuit which has
the same label as $pp^{-1}$ based at $\tilde{v}_0$. By uniqueness of
a path starting at $\tilde{v}_0$ with a given label,  this circuit
must be of the form $\tilde{p}\tilde{p}^{-1}$ for some lift
$\tilde{p}$ of the path $p$ starting at $\tilde{v}_0$.

(c) Let $p$ and $\tilde{p}$ be as in part (b) above, denote
$\omega(\tilde{p})$ by $\tilde{v}_1$ and let $H_1 =
f(L(\tilde{\Gamma},\tilde{v}_1))$. If $h \in H$, then $php^{-1} \in
H$ since $p \in N(H,v_0)$. This lifts to a circuit
$\tilde{p}\tilde{h}_1\tilde{p}^{-1}$ based at $\tilde{v}_0$ since
every circuit in $H$ lifts to a circuit in $\tilde{\Gamma}$ based at
$\tilde{v}_0$. This forces $\tilde{h}_1$ to be a circuit in
$\tilde{\Gamma}$ based at $\tilde{v}_1$ and with the same label as
$h$, and so $h = f(\tilde{h}_1) \in H_1$. Hence $H \subseteq H_1$.

Conversely, suppose that $h \in H_1$. Then $h$ lifts to a circuit
$\tilde{h}_1$ in $\tilde{\Gamma}$ based at $\tilde{v}_1$. So
$\tilde{p}\tilde{h}_1\tilde{p}^{-1} \in
L(\tilde{\Gamma},\tilde{v}_0)$, and so $php^{-1} \in H$. It follows
that $p^{-1}php^{-1}p \in H$ since $p \in N(H,v_0)$, so this lifts
to a circuit of the form
$\tilde{p}^{-1}\tilde{p}\tilde{h}\tilde{p}^{-1}\tilde{p}$ based at
$\tilde{v}_0$ for some path $\tilde{h}$ with the same label as $h$.
This path $\tilde{h}$ must be a circuit based at $\tilde{v}_0$ and
so $f(\tilde{h}) = h \in H$. Hence $H_1 \subseteq H$ and so $H_1 =
H$. By Lemma \ref{iso}  this implies that there is a deck
transformation $\phi_p$ of $\tilde{\Gamma}$ that maps $\tilde{v}_0$
onto $\tilde{v}_1$. The last statement of part (c) of the lemma is
immediate since automorphisms of an edge-labeled graph are
determined by where they send a point.
\end{proof}

The following theorem provides an analogue for graph immersions
 Theorem
\ref{deckcovers}.

\begin{Theorem} \label{deck}

Let $f : \tilde{\Gamma} \rightarrow \Gamma$ be an immersion of
connected graphs with  $\tilde{v}_0 \in f^{-1}(v_0)$ and $H =
f(L(\tilde{\Gamma},\tilde{v}_0))$. Then $G(\tilde{\Gamma}) \cong
N(H,v_0)/H$.

\end{Theorem}

\begin{proof}
Define $r : N(H,v_0) \rightarrow G(\tilde{\Gamma})$ by $r(p) =
{\phi}_{p}$ in the notation of Lemma \ref{normalizerlift2}. If $p,q
\in N(H,v_0)$, then by Proposition \ref{graphisomorphisms} the deck
transformation $\phi_p$ maps $\omega(\tilde{q})$ to
$\omega(\tilde{q}_1)$ where $\tilde{q}_1$ is a path starting at
$\omega(\tilde{p})$ with $\ell(\tilde{q}_1) = \ell(\tilde{q})$. So
$\phi_{pq}(\tilde{v}_0) = \omega(\tilde{p}\tilde{q}_1) =
\omega(\tilde{q}_1) = \phi_p(\omega(\tilde{q})) =
\phi_p(\phi_q(\tilde{v}_0))$, and so $\phi_{pq} = \phi_p\circ
\phi_q$. Hence $r$ is a homomorphism. It is surjective by Lemma
\ref{normalizerlift2}. Now $\tilde{\Gamma}$ is identified with the
graph of right $\omega$-cosets of $H$ in $FIM(X)$ and under this
identification, a vertex $\tilde{v}$ of $\tilde{\Gamma}$ is
identified with the right $\omega$-coset $(Hp)^{\omega}$ where $p$
lifts to a path $\tilde{p}$ in $\tilde{\Gamma}$ starting at
$\tilde{v}_0$ and ending at $\tilde{v}$. By Lemma
\ref{normalizerlift2}, $\phi_p = \phi_q$ iff $\omega(\tilde{p}) =
\omega(\tilde{q})$, so the kernel of the map $r$ coincides with the
equivalence relation $\rho_H$ that identifies $p$ and $q$ if
$(Hp)^{\omega} = (Hq)^{\omega}$. It follows that $G(\tilde{\Gamma})
\cong N(H,v_0)/H$.
\end{proof}

\section{Covers}

In this section we characterize covers of graphs in terms of the
concepts introduced earlier. A result related to part (a) of the
following theorem was obtained by Meakin and Szak\'acs \cite{MeSz}
in the more general context of immersions between $2$-complexes.

\begin{Theorem} \label{covers}

Let $f : \tilde{\Gamma} \rightarrow \Gamma$ be an immersion of
connected graphs. Choose basepoints $v_0 \in \Gamma$ and
$\tilde{v}_0 \in \tilde{\Gamma}$ such that $\tilde{v}_0 \in
f^{-1}(v_0)$ and let $H = f(L(\tilde{\Gamma},\tilde{v}_0))$. Then

\begin{enumerate}[label=(\alph*)]

\item $\tilde{\Gamma}$ is a cover of $\Gamma$ if and only if $H$ is a
full inverse submonoid of $L(\Gamma,v_0)$.


\item $\tilde{\Gamma}$ is a normal cover of $\Gamma$ if and only if
$N(H,v_0) = L(\Gamma,v_0)$
\item  $\tilde{\Gamma}$ is the universal cover of $\Gamma$ if and only
if $H$ consists of the idempotents in $L(\Gamma,v_0)$, i.e. $H = \{p
\in L(\Gamma,v_0) : \ell(p)$ is a Dyck word $\}$.

\end{enumerate}

\end{Theorem}
\begin{proof}

(a) Suppose that $\tilde{\Gamma}$ is a cover of graphs. Then every
path $p$ in $\Gamma$ that starts at $v_0$ lifts to a (unique) path
$\tilde{p}$ starting at $\tilde{v}_0$ by Proposition
\ref{pathlifting}. In particular, if $e$ is an idempotent of
$L(\Gamma,v_0)$, then $\ell(e)$ is a Dyck word since the only
idempotents in $FIC(\tilde{\Gamma})$ are Dyck words, so
$\ell(\tilde{e})$ is a path starting at $\tilde{v}_0$ whose label is
a Dyck word. This forces $\tilde{e}$ to be a circuit at
$\tilde{v}_0$, so $e = f(\tilde{e}) \in H$. Hence $H$ is full in
$L(\Gamma,v_0)$. Conversely, suppose that $H$ is full in
$L(\Gamma,v_0)$. Let $v_1$ be any vertex in $\Gamma$ and
$\tilde{v}_1$  any vertex in $f^{-1}(v_1)$ and let $p$ be a path in
$\Gamma$ starting at $v_1$. There is a path $\tilde{q}$ in
$\tilde{\Gamma}$ from $\tilde{v}_0$ to $\tilde{v}_1$ and the
projection of this path is a path $q$ in $\Gamma$ from $v_0$ to
$v_1$. Then $qpp^{-1}q^{-1}$ is an idempotent in $L(\Gamma,v_0)$ so
it  is in $H$ and hence it lifts to a path
$\tilde{q}_1\tilde{p}_1\tilde{p}_1^{-1}\tilde{q}_1^{-1}$ in
$\tilde{\Gamma}$ starting at $\tilde{v}_0$. Since $\ell(\tilde{q}_1)
= \ell(q) = \ell(\tilde{q})$ we must have $\tilde{q}_1 = \tilde{q}$.
Hence $\omega(\tilde{q}_1) = \tilde{v}_1$ and so $\tilde{p}_1$ is a
lift of $p$ that starts at $\tilde{v}_1$. Hence all paths in
$\Gamma$ lift everywhere, so $\tilde{\Gamma}$ is a cover of $\Gamma$
by Proposition \ref{pathlifting}.

(b) Suppose that $\tilde{\Gamma}$ is a normal cover of $\Gamma$ and
$v_0 \in \Gamma$. Then for all vertices $\tilde{v}_0,\tilde{v}_1 \in
f^{-1}(v_0)$ there is a (unique) deck transformation $\phi$ mapping
$\tilde{v}_0$ onto $\tilde{v}_1$. If $p \in L(\Gamma,v_0)$ then
$pp^{-1} $ is an idempotent in $L(\Gamma,v_0)$ so $pp^{-1} \in H$ by
part (a). Hence $pp^{-1}$ lifts to a path $\tilde{p}\tilde{p}^{-1}$
at $\tilde{v}_0$ and so $p$ lifts to the path $\tilde{p}$ from
$\tilde{v}_0$ to some vertex $\tilde{v}_1$. Since there is a deck
transformation $\phi$ that maps $\tilde{v}_0$ onto $\tilde{v}_1$,
part (a) of Lemma \ref{normalizerlift2} implies that $p \in
N(H,v_0)$.  Hence $L(\Gamma,v_0) \subseteq N(H,v_0)$ and so
$L(\Gamma,v_0) = N(H,v_0)$.

Conversely, suppose that $L(\Gamma,v_0) = N(H,v_0)$ and $e$ is an
idempotent in $L(\Gamma,v_0)$. Then $e \in N(H,v_0)$. But $H$ is
full in $N(H,v_0)$ by Proposition \ref{cosets}, so $e \in H$. Hence
$\tilde{\Gamma}$ is a cover of $\Gamma$ by part (a). Now let
$\tilde{v}_0, \tilde{v}_1 \in f^{-1}(v_0)$. There is some path
$\tilde{r}$ from $\tilde{v}_0$ to $\tilde{v}_1$ that projects to a
path $r \in L(\Gamma,v_0) = N(\Gamma,v_0)$. So by Lemma
\ref{normalizerlift2} there is a deck transformation that maps
$\tilde{v}_0$ to $\tilde{v}_1$, and so $\tilde{\Gamma}$ is a normal
cover of $\Gamma$.

(c) Suppose  that $\tilde{\Gamma}$ is the universal cover of
$\Gamma$. Then $\tilde{\Gamma}$ is a tree, so the label of any
circuit $\tilde{p}$ in $\tilde{\Gamma}$ based at any point is a Dyck
word. Such circuits project onto circuits in $\Gamma$ whose label is
also a Dyck word, so $H$ consists just of idempotents in
$L(\Gamma,v_0)$. On the other hand, if $e$ is an idempotent in
$L(\Gamma,v_0)$ then $\ell(e)$ is a Dyck word (since any idempotent
in $FIC(\Gamma)$ is a path whose label is a Dyck word). Since $e$
lifts to a path $\tilde{e}$ based at $\tilde{v}_0$ and since
$\ell(\tilde{e}) $ is a Dyck word, $\tilde{e}$ is a circuit based at
$\tilde{v}_0$, so $f(\tilde{e}) = e \in H$. Hence $H$ consists
exactly of all the Dyck words in $L(\Gamma,v_0)$. Conversely, if $H$
consists  of these idempotents, then $H$ is full in $L(\Gamma,v_0)$
so $\tilde{\Gamma}$ is a cover of $\Gamma$ by part (a). By
definition of $H$, any circuit in $\tilde{\Gamma}$ based at
$\tilde{v}_0$ projects onto a circuit in $H$ based at $v_0$, so its
label must be a Dyck word. It follows that $\tilde{\Gamma}$ is a
tree. Hence $\tilde{\Gamma}$ is the universal cover of $\Gamma$.
\end{proof}

\section {Actions of groups on graphs}

By an {\em action} of a group $G$ on a graph $\Gamma$ we mean a
homomorphism $\phi: G \rightarrow Aut(\Gamma)$ of $G$ into the group
of (labeled graph) automorphisms of $\Gamma$. Here we adopt the
convention that the action is a left action and denote the image
under $\phi(g)$ of the vertex $v$ [or edge $e$] by $g.v$ [resp.
$g.e$]. If $f : \tilde{\Gamma} \rightarrow \Gamma$ is an immersion
between connected graphs then there is an obvious action of the
group $G(\tilde{\Gamma})$ on $\tilde{\Gamma}$. The following fact is
well-known for {\em covers} of connected graphs, but also holds for
immersions (with essentially the same proof).

\begin{Prop}
\label{actions}

If $f : \tilde{\Gamma} \rightarrow \Gamma$ is an immersion of
connected graphs with group $G = G(\tilde{\Gamma})$ of deck
transformations, then the quotient map $g : \tilde{\Gamma}
\rightarrow \tilde{\Gamma}/{G}$ is a normal covering and $G$ is the
group of deck transformations of this cover.

\end{Prop}

\begin{proof}

If $\gamma$ is a deck transformation of the immersion $f$ then by
Proposition \ref{uniquemorph}, $\gamma$ is uniquely determined by
the image of any vertex in $\Gamma$, and it follows that if $v$ is a
vertex and $e$ is a  edge of $\Gamma$ and $\gamma$ is not the
identity automorphism of $\Gamma$, then $\gamma(v) \neq v$ and
$\gamma(e) \neq e$. From this it follows  that the action of $G$ on
$\tilde{\Gamma}$ satisfies condition (*) on page 72 of Hatcher
\cite{hatch} and so the result follows by Proposition 1.40 of
\cite{hatch}.
\end{proof}

\noindent {\bf Example } The graph $\tilde{\Gamma}$ in Figure 3 has
four vertices $v_1,v_2,v_3$ and $v_4$ and six positively labeled
edges. It immerses in the obvious way into the bouquet $B_{\{a,b\}}$
of two circles via the map $f$ that preserves edge labels.

 \begin{center}
\begin{tikzpicture}
    \begin{pgfonlayer}{nodelayer}
        \node [style=rn] (0) at (-1, -1) {$w_1$};
        \node [style=rn] (2) at (1, -1) {$w_2$};
        \node [style=rn] (3) at (-3, 5) {$v_1$};
        \node [style=rn] (4) at (-1, 5) {$v_2$};
        \node [style=rn] (6) at (1, 5) {$v_3$};
        \node [style=rn] (7) at (3, 5) {$v_4$};
        \node [style=none] (8) at (-5,-1) {$\Gamma$};
        \node [style=none] (9) at (-5,5) {$\tilde{\Gamma}$};
        \node [style=none] (10) at (0, 2) {};
        \node [style=none] (11) at (0, 0.25) {};

    \end{pgfonlayer}
    \begin{pgfonlayer}{edgelayer}
        \draw [style=red, in=135, out=-135, loop] (0) to node[auto, left]{a} ();
        \draw [style=red, in=45, out=-45, loop] (2) to node[auto, right]{a} ();
        \draw [style=red, bend right=45, looseness=1.50] (6) to node[auto,above]{a} (4);
        \draw [style=red, bend right=45, looseness=1.50] (4) to node[auto,below]{a} (6);
        \draw [style=red, bend right=45, looseness=1.50] (7) to node[auto,above]{a} (3);
        \draw [style=red, in=-135, out=-45, looseness=1.50] (3) to node[auto,below]{a} (7);
        \draw [style=blue] (3) to node[auto]{b} (4);
        \draw [style=blue] (7) to node[auto,above]{b} (6);
        \draw [style=blue] (0) to node[auto]{b} (2);
        \draw [style=black] (10) to node[auto]{g} (11);

    \end{pgfonlayer}
\end{tikzpicture}
\end{center}

\begin{center}
Figure 3
\end{center}

Let $e_1$ be the directed edge from $v_1$ to $v_4$ with label $a$;
$e_2$  the directed edge from $v_2$ to $v_3$ with label $a$; $e_3$
 the directed edge from $v_3$ to $v_2$ with label $a$; $e_4$
the directed edge from $v_4$ to $v_1$ with label $a$; $e_5$  the
directed edge from $v_1$ to $v_2$ with label $b$; and $ e_6$  the
directed edge from $v_4$ to $v_3$ with label $b$.

 The stabilizer
of the vertex $v_1$ under the action by $FIM(a,b)$ is the closed
inverse submonoid $H$ of $FIM(a,b)$ generated by the elements
$ba^2b^{-1}, aba^{-1}b^{-1}$ and $bab^{-1}a$. The graph
$\tilde{\Gamma}$ is the graph of right $\omega$-cosets of $H$ in
$FIM(a,b)$. The distinct right $\omega$-cosets of $H$ are of course
$H,(Hb)^{\omega}, (Hba)^{\omega}$ and $(Ha)^{\omega}$, and they may
be identified with the four vertices $v_1,v_2,v_3$ and $v_4$
respectively.

The map that interchanges  $e_1$ and $e_4$, interchanges $e_2$ and
$e_3$, and interchanges $e_5$ and $e_6$ defines a deck
transformation $\gamma$ of the immersion $f$ that interchanges $v_1$
with $v_4$ and $v_2$ with $v_3$. This is the only non-trivial deck
transformation, so the group $G = G(\tilde{\Gamma})$ is isomorphic
to the cyclic group ${\mathbb Z}_2$ of order $2$. Note that the only
right $\omega$-coset of $H$ that lies in $N(H)$ is $(Ha)^{\omega}$,
so $N(H)/H = \{H,(Ha)^{\omega}\}$ is isomorphic to ${\mathbb Z}_2$
in accord with Theorem \ref{deck}.

The group $G = {\mathbb Z}_2$ acts on $\tilde{\Gamma}$ in the
obvious way and the quotient graph $\Gamma = \tilde{\Gamma}/G$ is
the graph with two vertices $w_1,w_2$ and one positively oriented
edge from $w_1$ to $w_2$ with label $b$, as depicted in Figure 3.
Clearly the map $g$ from $\tilde{\Gamma}$ to $\Gamma$ that preserves
edge labels is a normal cover of $\Gamma = \tilde{\Gamma}/G$, in
accord with Proposition \ref{actions}. \QED

\section{Extending immersions to covers}

 \begin{Theorem}
 \label{restrictions}
 Let $f : \tilde{\Gamma} \rightarrow \Gamma$ be an immersion of
  connected graphs. Then there is a graph cover $g : \tilde{\Delta}
 \rightarrow \Gamma$ such that

 (a)  $\tilde{\Gamma}$ is a subgraph of
 $\tilde{\Delta}$ and $f$ is the restriction of $g$ to
 $\tilde{\Gamma}$; and

 (b) any deck transformation of $\tilde{\Gamma}$ is the restriction
 of some deck transformation of $\tilde{\Delta}$.

 \end{Theorem}

 \begin{proof}

 (a) If $f : \tilde{\Gamma} \rightarrow \Gamma$ is  a cover of graphs
there is nothing to prove, so assume that this is not the case. Let
the edges of $\Gamma$ and $\tilde{\Gamma}$ be labeled over a set $X
\cup X^{-1}$ consistent with an immersion into $B_X$ as usual. Then
there is a vertex  $\tilde{v}$ of $\tilde{\Gamma}$ for which there
is an edge $e$ in $\Gamma$ with $f(\tilde{v}) = \alpha(e)$ such that
$e$ does not lift to any edge in $\tilde{\Gamma}$ starting at
$\tilde{v}$. If $\ell(e) = x \in X \cup X^{-1}$ then there is no
edge with label $x$ starting at $\tilde{v}$. We refer to such a
vertex as an {\em incomplete vertex} of $\tilde{\Gamma}$ and say
that $\tilde{v}$ is {\em missing an edge labeled by} $x$. Enlarge
the graph $\tilde{\Gamma}$ by adding a new vertex $\tilde{v}_x$ and
a new edge $\tilde{e}_x$ from $\tilde{v}$ to $\tilde{v}_x$ for each
incomplete vertex $\tilde{v}$ of $\tilde{\Gamma}$ that is missing an
edge labeled by $x$. Since distinct edges starting at $f(\tilde{v})$
have distinct labels, the new vertices and edges that we added to
$\tilde{\Gamma}$ are all distinct. Clearly the graph
$\tilde{\Delta}_1$ immerses into $\Gamma$ via the map that preserves
edge labeling. Then apply the same process to $\Delta_1$, adding new
edges as necessary at any incomplete vertices, to form the graph
$\Delta_2$. Continue in this fashion to build a sequence of graphs
$\tilde{\Gamma} \subseteq \tilde{\Delta}_1 \subseteq
\tilde{\Delta}_2 \subseteq ...$ by adding new vertices and edges to
the previous graph at any incomplete vertices. Let $\tilde{\Delta}$
be the union of the graphs $\tilde{\Delta}_i$ as $ i$ ranges from
$1$ to $\infty$. The graph $\tilde{\Delta}$ is obtained from
$\tilde{\Gamma}$ by adding (possibly infinite) trees to incomplete
vertices of $\tilde{\Gamma}$. Then the map $g : \tilde{\Delta}
\rightarrow \Gamma$ that extends $f$ and maps paths in
$\tilde{\Delta} \setminus \tilde{\Gamma}$ to their obvious images in
$\Gamma$ is a covering map since $\tilde{\Delta}$ has no incomplete
vertices.

(b) Suppose now that $\gamma$ is a deck transformation of the
immersion $f : \tilde{\Gamma} \rightarrow \Gamma$ that takes a
vertex $\tilde{v}$ to a vertex $\tilde{v}_1$. Any path $\tilde{p}$
in $\tilde{\Delta}$ starting at $\tilde{v}$ factors in the form
$\tilde{p} = \tilde{p}_1\tilde{q}_1\tilde{p}_2\tilde{q}_2
...\tilde{p}_n\tilde{q}_n$ where
$\tilde{p}_1\tilde{p}_2...\tilde{p}_n$ is a path in $\tilde{\Gamma}$
and the $\tilde{q}_i$ are circuits in the forest $\tilde{\Delta}
\setminus \tilde{\Gamma}$ based at $\omega(\tilde{p}_i)$ for $i =
1,...,n-1$ and $\tilde{q}_n$ is a path in $\tilde{\Delta} \setminus
\tilde{\Gamma}$ starting at $\omega(\tilde{p}_n)$. By Proposition
\ref{graphisomorphisms}, there is a path $\tilde{p}_1' $ in
$\tilde{\Gamma}$ starting at $\tilde{v}_1$ with $\ell(\tilde{p}_1')
= \ell(\tilde{p}_1)$. Also by the same proposition, there is an edge
$\tilde{e}$ in $\tilde{\Gamma}$ starting at $\omega(\tilde{p}_1)$
with label $x$ if and only if there is an edge $\tilde{e}'$ in
$\tilde{\Gamma}$ starting at $\omega(\tilde{p}')$ with the same
label. It follows that there is a path of the form $\tilde{q}_1'$ in
$\tilde{\Delta} \setminus \tilde{\Gamma}$ with $\ell(\tilde{q}_1') =
\ell(\tilde{q}_1)$, $\alpha(\tilde{q}_1') = \omega(\tilde{p}_1')$
and $\tilde{q}_i'$ is a circuit if and only if $\tilde{q}_i$ is a
circuit. Continuing in this fashion, we see that there is a path
$\tilde{p}' = \tilde{p}_1'\tilde{q}_1'\tilde{p}_2'\tilde{q}_2'
...\tilde{p}_n'\tilde{q}_n'$ with $\ell(\tilde{p}') =
\ell(\tilde{p})$. Furthermore, $\tilde{p}'$ is a circuit if and only
if $\tilde{p}$ is a circuit since $\omega(\tilde{p}') =
\gamma(\omega(\tilde{p}))$. Hence by Proposition
\ref{graphisomorphisms}  there is a deck transformation
$\tilde{\gamma}$ of $\tilde{\Delta}$ that takes $\tilde{v}$ to
$\tilde{v}_1$. The bijection $\tilde{\gamma}$ extends $\gamma$ and
maps a path in $\tilde{\Delta} \setminus \tilde{\Gamma}$ starting at
an incomplete vertex $\tilde{w}$ to the path with the same label
starting at $\gamma(\tilde{w})$.
\end{proof}

We remark that it is not clear that it is possible to extend {\em
finite} immersions between graphs to {\em finite} covers in such a
way that deck transformations of the immersion are restrictions of
deck transformations of the cover.

\medskip

\noindent {\bf Acknowledgement.} We are grateful to Nora Szak\'acs
for pointing out an error in an earlier version of this paper.

\medskip

\noindent Department of Mathematics, University of Nebraska,
Lincoln, Nebraska 68588, USA.

\noindent John Meakin   \, \, jmeakin@math.unl.edu;

\noindent  Corbin Groothuis \, \, corbin.groothuis@huskers.unl.edu

\end{document}